\documentclass[12pt]{article}

\usepackage{diagbox}
\usepackage{mathtools}
\usepackage{bbm}
\usepackage{latexsym}
\usepackage{epsfig}
\usepackage{subfig}
\usepackage{amsmath,amsthm,amssymb,enumerate}

\usepackage[a-1b]{pdfx}
\usepackage{hyperref}

\usepackage{color}

\def\nn{\nonumber}
\def\a{\alpha}  \def\d{\delta} \def\D{\Delta}
\def\e{\varepsilon}    \def\g{\gamma}
\def\G{\Gamma}  
 \def\th{\theta}    
  \def\n{\nu} 
  \def\s{\sigma} 
 \def\om{\omega}

\def\bx{{\bf x}}

\newtheorem{theorem}{Theorem}
\newtheorem{lemma}[theorem]{Lemma}
\newtheorem{corollary}[theorem]{Corollary}
\newtheorem{Remark}{Remark}

\newcommand{\brac}[1]{\left(#1\right)}

\newcommand{\bfrac}[2]{\left(\frac{#1}{#2}\right)}

\allowdisplaybreaks
\newcommand{\set}[1]{\left\{#1\right\}}
\def\sm{\setminus}
\def\seq{\subseteq}

\def\E{\mathbb{E}}

\def\Pr{\mathbb{P}}

\newcommand{\ignore}[1]{}

\newcommand{\beq}[2]{\begin{equation}\label{#1}#2\end{equation}}
\newcommand{\mults}[1]{\begin{multline*}#1\end{multline*}}

\def\nn{\nonumber}

\usepackage{tikz}
\usetikzlibrary{decorations.pathmorphing}
\usetikzlibrary{decorations.pathreplacing}
\usetikzlibrary{positioning}
\usetikzlibrary{arrows,automata}
\usetikzlibrary{shapes.misc}
\usetikzlibrary{backgrounds}
\usetikzlibrary{arrows,shapes}

\def\N{\mathbb{N}}

\begin{document}
	
\author{Alan Frieze\thanks{Research supported in part by NSF grant DMS1952285} \ and Aditya Raut\\
	Department of Mathematical Sciences\\
	Carnegie Mellon University\\
	Pittsburgh PA-15213
	}

\date{}

\title{List chromatic number of the square of a sparse random graph}

\maketitle

\begin{abstract}
	We show that w.h.p the list chromatic number $\chi_\ell$ of the square of $G_{n,p}$ for $p=c/n$ is asymptotically equal to the maximum degree $\Delta(G_{n,p})$. Since $\chi(G^2_{n,p})\leq \chi_\ell(G^2_{n,p})$, this also improves an earlier result of Garapaty et al  \cite{KLMP} who proved that $\chi(G^2_{n,p}) \leq 6 \cdot \Delta(G_{n,p})$ w.h.p.
\end{abstract}

\section{Introduction}

The Erd\H{o}s-R\'enyi random graph $G_{n,p}$ for a positive integer $n$ and a real number $p\in [0,1]$ is defined as an $n$-vertex graph in which each pair of vertices $\{u,v\}$ is connected by an edge $uv$ with probability $p$, independently of all other pairs. Let $p=c/n$ where $c>0$ is a constant. The chromatic number of $G_{n,p}$ is well-understood, at least for sufficiently large $c$. {\L}uczak \cite{L} proved that if $G=G_{n,p}$ then $\chi(G)\sim \frac{c}{2\log c}$. This was refined by Achlioptas and Naor \cite{AN}, and further improved later by Coja-Oghlan and Vilenchik \cite{CV}. The list chromatic number $\chi_\ell$ is defined as the smallest number $k$ such that if each vertex is assigned a list of $k$ colors, there is always a valid proper coloring from those lists. Alon, Krivelevich and Sudakov \cite{AKS} showed that $\chi_\ell(G_{n,p}),p=c/n$ is $\Theta\bfrac{c}{2\log c}$. 

The square of a graph $G$ is obtained from $G$ by adding edges for all pairs of vertices at distance two from each other. Atkinson and Frieze \cite{AF} showed that w.h.p. the independence number of $G_2=G_{n,p}^2$ is asymptotically equal to $\frac{4n\log c}{c^2}$, for large $c$. Garapaty, Lokshtanov, Maji and Pothen \cite{KLMP} studied the chromatic number of powers of $G_{n,p}$. Let $\D=\D(G_{n,p})\sim\tfrac{\log n}{\log\log n}$ be the maximum degree in $G=G_{n,p}$ (for a proof of this known claim about the maximum degree, see for example \cite{FK}, Theorem 3.4). Garapaty et al proved, in the case of the square $G_2$ of $G_{n,p}$ with $p=c/n$, that $\chi(G_2) \leq 6\cdot \frac{\log n}{\log\log n}$ w.h.p. In this work, we strengthen this bound and prove a more general result about the list chromatic number. 

\begin{theorem}\label{th1}
	Let $p=c/n$ where $c>0$ is a constant. Let $G_2$ denote the square of $G_{n,p}$. Then w.h.p., $\chi(G_2)\sim \chi_\ell(G_2)\sim\D(G_{n,p})\sim\frac{\log n}{\log\log n}$.
\end{theorem}

We show that w.h.p. $G_2$ is $q$-list-colorable with $q = (1+3\th^{1/3})\D$ where $\th = o(1)$ is given in \eqref{DV}, establishing the upper bound. Note that for every graph $G$, $\chi(G)\leq \chi_\ell(G)$. Since the neighbors of a vertex in $G_{n,p}$ form a clique in the square graph $G_2$, the lower bound of $\D$ in the theorem is trivial. 

\begin{Remark}
	The value of $c$ does not contribute to the main term in the claim of Theorem \ref{th1}. Thus we would expect that we could replace $p=c/n$ by $p\leq \om/n$ for some slowly growing function $\om=\om(n)\to \infty$. Indeed, a careful examination of the proof below verifies this so long as $c=o(\log\log n)$.
\end{Remark}

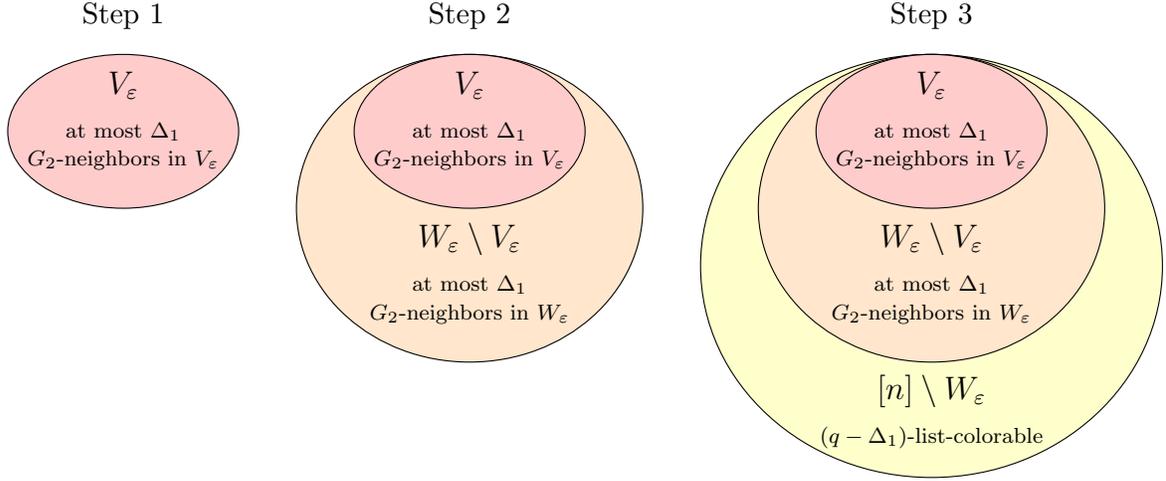
\begin{figure}[htbp]
	\centering
	
	\resizebox{0.8\textwidth}{!}
	{
		
		\begin{tikzpicture}[scale=0.5]			
			
			% Ellipses for the sets
			\draw[fill=yellow!20] (0,-3.5) ellipse (6 and 5.5);
			\draw[fill=orange!20] (0,-2) ellipse (4.5 and 4);
			\draw[fill=red!20]    (0, 0) ellipse (3 and 2);
			
			\draw[fill=orange!20] (-12,-2) ellipse (4.5 and 4);
			\draw[fill=red!20]    (-12, 0) ellipse (3 and 2);
			
			\draw[fill=red!20] (-21,0) ellipse (3 and 2);	
			
			% Labels for the sets
			\node at (0, 1.2) {$V_\varepsilon$};
			\node at (0,-2.8) {$W_\varepsilon \setminus V_\varepsilon$};
			\node at (0,-6.75) {$[n] \setminus W_\varepsilon$};
			
			\node at (-12,1.2) {$V_\varepsilon$};
			\node at (-12,-2.8) {$W_\varepsilon \setminus V_\varepsilon$};
			
			\node at (-21,1.2) {$V_\varepsilon$};
			
			% Annotations inside ellipses
			\node at (0,     0) {\scriptsize at most $\Delta_1$};
			\node at (0, -0.75) {\scriptsize $G_2$-neighbors in $V_\varepsilon$};	
			\node at (0,    -4) {\scriptsize at most $\Delta_1$};
			\node at (0, -4.75) {\scriptsize $G_2$-neighbors in $W_\varepsilon$};
			\node at (0, -7.95) {\scriptsize $(q-\D_1)$-list-colorable};
			
			\node at (-12,     0) {\scriptsize at most $\Delta_1$};
			\node at (-12, -0.75) {\scriptsize $G_2$-neighbors in $V_\varepsilon$};	
			\node at (-12,    -4) {\scriptsize at most $\Delta_1$};
			\node at (-12, -4.75) {\scriptsize $G_2$-neighbors in $W_\varepsilon$};
			
			\node at (-21,     0) {\scriptsize at most $\Delta_1$};
			\node at (-21, -0.75) {\scriptsize $G_2$-neighbors in $V_\varepsilon$};	
			
			%Step labels
			
			\node at (-21, 3) {\small Step 1};
			\node at (-12, 3) {\small Step 2};
			\node at (  0, 3) {\small Step 3};
			
		\end{tikzpicture}
		
	}
	
	\caption{The set of colored vertices in $G_2$ at the end of each step}
	\label{fig1}
	
\end{figure}

\subsection{Overview of the proof}

The main idea of the proof is to color the vertices $[n]$ of the square graph $G_2$ by dividing them into parts and assigning colors in a carefully chosen order greedily, as shown in Figure \ref{fig1}, where $\D_1\sim\D$. When assigning color to a vertex $v$, we ensure that the number of already colored neighbors of $v$ in $G_2$ is less than $q$. We can then use a greedy coloring strategy to obtain the claimed upper bound $q$ on the list chromatic number. 

From here onward, we use `neighbors' and `degree' specifically for the neighbors and degree of a vertex in $G_{n,p}$. Similarly, we use `$G_1$-neighbors' and `$G_1$-degree' when $G_1 = G_{n,p}$. We specify `$G_2$-neighbors' or `$G_2$-degree' when referring to the neighbors or the degree of a vertex in the square graph. We define $V_\e$ as the set of vertices of `high' degree (at least $\e$ fraction of maximum degree) and $W_\e$ as their closed neighborhood, for a carefully chosen $\e$. In particular, the aforementioned order of coloring vertices is: \vspace{-10pt}

\begin{enumerate}[\qquad Step 1.]
	\setlength\itemsep{0pt}
	\item All the vertices with a high degree ($V_\e$)
	\item Neighbors of all high-degree vertices ($W_\e \sm V_\e$)
	\item Remaining vertices ($[n]\sm W_\e$)
\end{enumerate} \vspace{-10pt}

We bound the number of $G_2$-neighbors of $V_\e$ within $V_\e$ by $\D_1 =\brac{1+2\th^{1/3}}\D$ in Corollary \ref{cor1}, ensuring step 1 of the coloring. The number of $G_2$-neighbors of $W_\e \sm V_\e$ within $W_\e$ is bounded by $\D_1$ in Corollary \ref{cor2}, which ensures the completion of step 2. For step 3, we prove that $G_2$ restricted to $[n]\sm W_\e$ can be list-colored with a small number of colors in Corollary \ref{cor3}. This number of extra colors along with the bound on the number of $G_2$-neighbors of $[n]\sm W_\e$ in $W_\e$ given by Corollary \ref{cor2} is smaller than $q$, completing step 3. 

Corollaries \ref{cor1} and \ref{cor2} are proved using the structural results of Lemmas \ref{lem1} and \ref{lem2}. The list-coloring claim for $[n]\sm W_\e$ in Corollary \ref{cor3} is proved by establishing that any subset $S\subseteq [n]\sm W_\e$ contains at most $(6+2c)\e\D|S|$ edges in $G_2$. For `large' sets $S$, this is proved in Lemma \ref{lem3} and `small' sets are handled by Lemmas \ref{lem4} and \ref{lem5}.

\subsection{Organization of the paper}

In Section \ref{2}, we provide the proofs of Corollaries \ref{cor1}, \ref{cor2} and \ref{cor3} that are required for the three steps of the coloring, and mention statements of all the supporting lemmas used in these proofs. Section \ref{sec3} contains the explicit details of the greedy coloring strategy and proofs of all the lemmas mentioned in Section \ref{2}. We conclude with Section \ref{4}, which mentions some directions for future work and our remarks on the result.

\section{Proof of Theorem \ref{th1}}\label{2}

Let $d(v)$ and $N(v)$ denote the degree and neighborhood of $v$ in $G_{n,p}$ respectively, and let $\D$ be the maximum degree of a vertex.

We can use the following high probability bounds for $\D$ taken from  \cite{FK}, Theorem 3.4:
\[
\frac{\log n}{\log\log n}\brac{1-\frac{3\log\log\log n}{\log\log n}}\leq \D\leq \frac{\log n}{\log\log n}\brac{1+\frac{3\log\log\log n}{\log\log n}}
\]
This implies that w.h.p.
\beq{DV}{
n^{1-\th}\leq \D^\D\leq n^{1+\th}\text{ where }\th=\frac{4\log\log\log n}{\log\log n}.
}

For the above value of $\th = o(1)$, we fix $$\e=\th^{1/2}.$$ 

For each $0<\a\leq 1$, we define $V_\a=\set{v:d(v)\geq \a\D}$ as the set of vertices with degree at least an $\alpha$ fraction of the maximum degree. Let $W_\a$ denote the closed neighborhood of $V_\a$, i.e., the neighbors of $V_\a$ in $G_{n,p}$ along with the vertices $V_\a$. A subset of vertices in $V_\e$ with sum of degrees comparable to our bound $q$ will be of interest. Thus, define a set of `good' $m$-tuples of degrees as
\[
L_m=\set{(\ell_1,\ell_2,\ldots,\ell_m)\in \set{\e\D,\e\D+1,\ldots,\D}^m\;:\sum_{i=1}^m\ell_i\geq (1+\th^{1/3})\D}.
\]

\subsection{Bounding the number of \texorpdfstring{$G_2$}{}-neighbors in \texorpdfstring{$W_\e$}{}}
The following two lemmas are needed to analyze the coloring of vertices in $W_\e$. We prove them in Section \ref{sec3}.

\begin{lemma}\label{lem1}
	W.h.p., $v,w\in V_{2/3}$ implies that $dist(v,w)\geq 10$. (Here $dist(.,.)$ is graph distance in $G_{n,p}$.)
\end{lemma}

\begin{lemma}\label{lem2}
	Suppose that $m\leq 2/\e$. Then w.h.p. there does not exist a connected subset $S\subseteq [n]$ of $G_{n,p}$ with at  most $3m$ vertices, which contains an $m$-sized subset of vertices $w_i$ with a `good' $m$-tuple of degrees, i.e., $(d(w_i),i=1,2,\ldots,m)\in L_m$.
\end{lemma}

\begin{corollary}\label{cor1}
	A vertex $v\in[n]$ has at most $\D_1 =\brac{1+2\th^{1/3}}\D$ $G_2$-neighbors in $V_\e$, w.h.p.
\end{corollary}
\begin{proof}
	Suppose $v$ has more than $\D_1$ $G_2$-neighbors in $V_\e$. Let $T$ be the tree obtained by Breadth-First Search to depth two from $v$ in $G_1=G_{n,p}$. We may assume this is a tree by ignoring the edges revisiting an explored vertex at depth two, if it has multiple parents at depth one. Remove all leaves from $T$ that are not in $V_\e$ and repeat, as shown in Figure \ref{fig2a}. We are left with a set of $G_1$-neighbors $W_0$ of $v$ in $V_\e$ and set of $G_1$-neighbors $u_1,u_2,\ldots,u_k$ of $v$ that are not in $V_\e$. In addition we have sets $W_1,W_2,\ldots,W_k\subseteq V_\e$ such that $u_i$ is a $G_1$-neighbor of all vertices in $W_i,i=1,2,\ldots,k$. The $G_2$-degree of $v$ is given by $D=\sum_{w\in W_0}d(w)+k+\sum_{i=1}^k|W_i|> \D_1$. The number of $G_2$-neighbors of $v$ in $V_\e$ is $m_1=\sum_{i=0}^k|W_i|$ and the tree $T$ contains $1+k+m_1\leq 2m_1+1$ vertices. Let $W=\bigcup_{i=0}^kW_i$ and add $v$ to $W$ if $v\in V_\e$. Then let $W=\{w_1,w_2,\ldots,w_m\}$ where $m=m_1+\mathbf{1}_{v\in V_\e}$. If $m\leq 2/\e$ then by letting $T$ take the place of $S$ in Lemma \ref{lem2}, we reach a contradiction. Otherwise, $M=\sum_{i=1}^m d(w_i)\geq md_{\min}$ where $d _{min}=\min\set{d(w):w\in W}$. But $d_{min}\geq \e\D$ and so $M>2\D$. It follows from Lemma \ref{lem1} that $d_{min}<2\D/3$ and so we can reduce $m$ by one, keeping $M>4\D/3$. If $m>2/\e$ after this update, we rewrite the expression for $M$ and repeat the argument. We eventually reduce $m$ to below $2/\e$ while keeping $M>(1+\th^{1/3})\D$. But now we contradict Lemma \ref{lem2} as discussed before.
\end{proof}

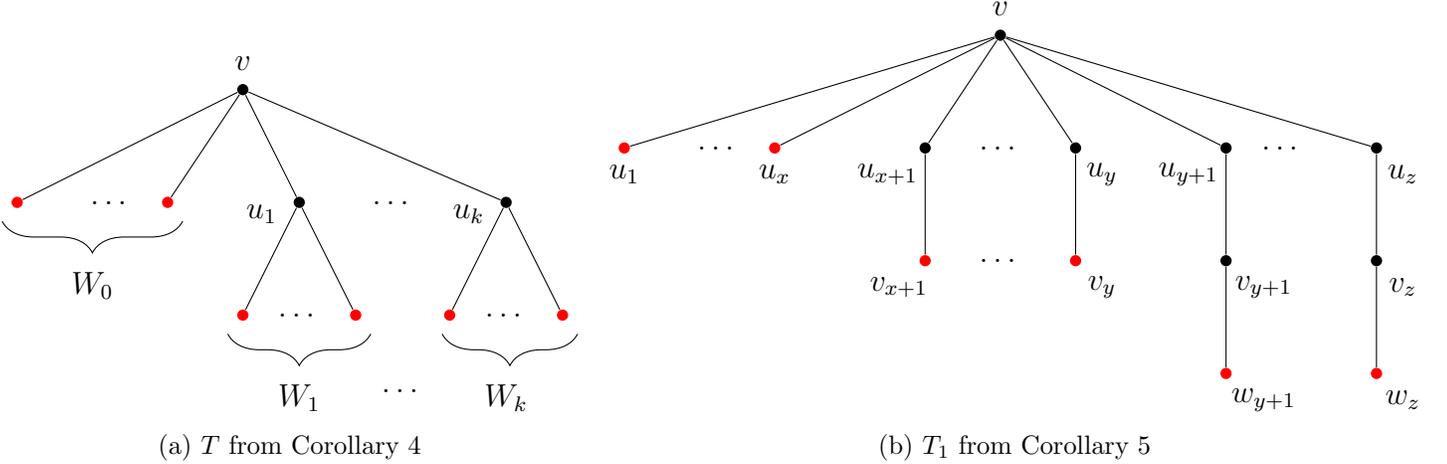
\begin{figure}[htbp]
	\centering 
	\subfloat[$T$ from Corollary \ref{cor1}]{
		\begin{tikzpicture}
			\node[fill=black, circle, inner sep=1.5pt] (v) at (0,2.5) {};
			\node at (0,2.85) {$v$};
			\node[fill=red, circle, inner sep=1.5pt] (w1) at (-3,1) {};
			\node[fill=red, circle, inner sep=1.5pt] (w2) at (-1,1) {};
			\node at (-1.75,1) {$\dots$};
			\node at (-2,-0.1) {$W_0$};
			\node[fill=black, circle, inner sep=1.5pt] (u1) at (0.75,1) {};
			\node at (0.25,0.85) {$u_1$};
			\node[fill=black, circle, inner sep=1.5pt] (uk) at (3.5,1) {};
			\node at (3,0.85) {$u_{k}$};
			\node at (2,1) {$\dots$};
			\node[fill=red, circle, inner sep=1.5pt] (v1) at (0,-0.5) {};
			\node[fill=red, circle, inner sep=1.5pt] (v2) at (1.5,-0.5) {};
			\node at (0.75,-0.5) {$\dots$};
			\node at (0.75,-1.6) {$W_1$};
			\node[fill=red, circle, inner sep=1.5pt] (vk1) at (2.75,-0.5) {};
			\node[fill=red, circle, inner sep=1.5pt] (vk2) at (4.25,-0.5) {};
			\node at (3.5,-0.5) {$\dots$};
			\node at (3.5,-1.6) {$W_{k}$};
			\node at (2.125,-1.5) {$\dots$};
			
			\draw (v) -- (w1);
			\draw (v) -- (w2);
			\draw (v) -- (u1);
			\draw (v) -- (uk);
			\draw (u1) -- (v1);
			\draw (u1) -- (v2);
			\draw (uk) -- (vk1);
			\draw (uk) -- (vk2);
			
			\draw [decorate,decoration={brace,amplitude=12}] (-0.8,0.75) -- (-3.2,0.75) ;
			\draw [decorate,decoration={brace,amplitude=12}] (1.7,-0.75) -- (-0.2,-0.75);
			\draw [decorate,decoration={brace,amplitude=12}] (4.45,-0.75) -- (2.65,-0.75) ;
			
		\end{tikzpicture}
		\label{fig2a}
	}	
	\subfloat[$T_1$ from Corollary \ref{cor2}]{
		\begin{tikzpicture}
			\node[fill=black, circle, inner sep=1.5pt] (v) at (0,2.5) {};
			\node at (0,2.85) {$v$};
			\node[fill=red, circle, inner sep=1.5pt] (u1) at (-5,1) {};
			\node at (-5,0.65) {$u_1$};
			\node at (-3.75,1) {$\dots$};
			\node[fill=red, circle, inner sep=1.5pt] (ux) at (-3,1) {};
			\node at (-3,0.65) {$u_x$};
			\node[fill=black, circle, inner sep=1.5pt] (ux+1) at (-1,1) {};
			\node at (-1.5,0.65) {$u_{x+1}$};
			\node at (0,1) {$\dots$};
			\node[fill=black, circle, inner sep=1.5pt] (uy) at (1,1) {};
			\node at (1.35,0.65) {$u_y$};
			\node[fill=black, circle, inner sep=1.5pt] (uy+1) at (3,1) {};
			\node at (2.5,0.65) {$u_{y+1}$};
			\node at (3.75,1) {$\dots$};
			\node[fill=black, circle, inner sep=1.5pt] (uz) at (5,1) {};
			\node at (5.35,0.65) {$u_z$};
			\node[fill=red, circle, inner sep=1.5pt] (vx+1) at (-1,-0.5) {};
			\node at (-1.35,-0.85) {$v_{x+1}$};
			\node at (0,-0.5) {$\dots$};
			\node[fill=red, circle, inner sep=1.5pt] (vy) at (1,-0.5) {};
			\node at (1.35,-0.85) {$v_y$};
			\node[fill=black, circle, inner sep=1.5pt] (vy+1) at (3,-0.5) {};
			\node at (3.5,-0.85) {$v_{y+1}$};
			\node[fill=black, circle, inner sep=1.5pt] (vz) at (5,-0.5) {};
			\node at (5.35,-0.85) {$v_z$};
			\node[fill=red, circle, inner sep=1.5pt] (wy+1) at (3,-2) {};
			\node at (3.5,-2.35) {$w_{y+1}$};
			\node[fill=red, circle, inner sep=1.5pt] (wz) at (5,-2) {};
			\node at (5.35,-2.35) {$w_z$};
			
			\draw (v) -- (u1);
			\draw (v) -- (ux);
			\draw (v) -- (ux+1);
			\draw (v) -- (uy);
			\draw (v) -- (uy+1);
			\draw (v) -- (uz);
			\draw (ux+1) -- (vx+1);
			\draw (uy) -- (vy);
			\draw (uy+1) -- (vy+1);
			\draw (uz) -- (vz);
			\draw (wy+1) -- (vy+1);
			\draw (wz) -- (vz);
			
		\end{tikzpicture}
		\label{fig2b}
	}
	\caption{Breadth-First-Search trees after deletions (red nodes are in $V_\e$).}
	\label{fig2}
\end{figure}

A similar argument bounds the number of $G_2$-neighbors in $W_\e$ for vertices outside $V_\e$.

\begin{corollary}\label{cor2}
	A vertex $v\notin V_\e$ has at most $\D_1 =\brac{1+2\th^{1/3}}\D$ $G_2$-neighbors in $W_\e$, w.h.p.
\end{corollary}
\begin{proof}
	Suppose $v$ has more than $\D_1$ $G_2$-neighbors in $W_\e$. Let $T$ be the tree obtained by Breadth-First Search to depth three from $v$ in $G_1=G_{n,p}$. Again, assuming this to be a tree involves ignoring the edges revisiting an already explored vertex, if it has multiple parents at smaller depths. Let this tree have levels $L_0=\set{v},L_1,L_2,L_3$.  Let the $G_1$-neighbors of $v$ be $\set{u_1,u_2,\ldots,u_k}$. Let $F_{i,t}$ for $t=2,3$ denote the vertices in $L_ t$ separated from $v$ in $T$ by $u_i$.
	
	We now define a subtree $T_1$ of $T$ that will take the place of $S$ in Lemma \ref{lem2}. To obtain $T_1$ we do the following: suppose that $u_1,u_2,\ldots,u_x$ are the neighbors of $v$ in $V_\e$. Delete $F_{i,2} \cup F_{i,3}$ from $T$ for $i\in [1,x]$. Now suppose that $X_i=F_{i,2}\cap V_\e\neq \emptyset$ for $i\in[x+1,y]$ and that $F_{i,2}\cap V_\e= \emptyset$ for $i\in [y+1,k]$. Choose one vertex $v_i\in X_i$ for each $i\in [x+1,y]$ and delete $X_i\setminus \set{v_i}$ and their children from $T$. Suppose also that $Y_i=F_{i,3}\cap V_\e\neq \emptyset$ for $i\in[y+1,z]$. Choose one vertex $v_i\in X_i$ along with one neighbor $w_i$ in $Y_i$ for each $i\in [y+1,z]$ and delete $X_i\setminus \set{v_i}$, $Y_i\setminus \set{w_i}$, and their children from $T$. For $i\in [z+1,k]$, we delete $u_i$ and the vertices $F_{i,2}\cup F_{i,3}$ from $T$. As shown in Figure \ref{fig2b}, $T_1$ is the tree that survives these deletions. All the leaves of $T_1$ are in $V_\e$. 
	
	The number of $G_2$-neighbors of vertex $v$ in $W_\e$ is at most \vspace{-2.75pt}
	$$ D=\sum_{i=1}^z d(u_i)+(k-z)\leq \sum_{i=1}^x d(u_i)+(z-x)\e\D+(k-z)\leq \sum_{i=1}^x d(u_i)+(z-x+1)\e\D. $$
	Our assumption is that $D > \D_1$. Now let $M=\sum_{w\in V(T_1)\cap V_{\e}}d(w)$. Then, \vspace{-2.75pt}
	$$M\geq \sum_{i=1}^x d(u_i)+(z-x)\e\D\geq D-\e\D>(1+2\th^{1/3}-\th^{1/2})\D>(1+\th^{1/3})\D.$$
	This is due to the fact that we get a contribution of at least $\e\D$ from each surviving member of $F_{i,2}$ and $F_{j,3}$ for $i\in[x+1,y]$ and $j\in[y+1,z]$. The number of vertices $N$ in the tree $T_1$ satisfies
	$$N\leq 1+x+2(y-x)+3(z-y)\leq 3|V(T_1)\cap V_\e|.$$
	Put $m=|V(T_1)\cap V_\e|$. If $m\leq 2/\e$, the tree $T_1$ contradicts Lemma \ref{lem2}. Assume then that $m>2/\e$, hence $M \geq m\e\D>2\D$. If the smallest degree in $V(T_1)\cap V_\e$ is at least $2\D/3$, then it follows from Lemma \ref{lem1} that $T$ consists only of $u_1\in V_{2/3}$ and the neighbors of $u_1$. In this case, the corollary holds trivially. Otherwise, we can delete a vertex of degree less than $2\D/3$ and reduce $m$ by one, keeping $M>(1+\th^{1/3})\D$. Similar to the proof of Corollary \ref{cor1}, this leads to a contradiction.
\end{proof}

\subsection{Extra colors required for \texorpdfstring{$[n]\sm W_\e$}{}}
The next part of our strategy is to bound the number of $G_2$-edges contained in any set $S$ that is disjoint from $W_\e$. We prove a high probability bound of $(6+2c)\e\D|S|$, which will imply Corollary \ref{cor3} as desired. 

For `large' sets $S$, the following lemma can be invoked for the desired bound. 
\begin{lemma}\label{lem3}
	The total number of edges in $G_2$ is less than $c(c+1)n$ w.h.p.
\end{lemma}
The proof of Lemma \ref{lem3} can be found in Section \ref{sec3}.

For $2\leq s\leq n$ let $\n_s$ be the maximum number of $G_2$-edges in a set of size $s$. 
\beq{eq1}{
	\text{If }k_0=2/\e^2\text{ then }\n_s\leq 10k_0c^3s\text{ for } s\geq n/(10ck_0).
}
In particular, the above statement derived using Lemma \ref{lem3} makes the notion of `large' precise for sets $S$. For further analysis, we restrict to the case of $|S| \leq n/(10ck_0) = n\e^2/(20c)$.

Let $e(S)$ denote the number of edges inside $S$ in $G_{n,p}$. We first show that the expected number of sets $S$ with more than $\frac{\e\D}{2}|S|$ edges is $o(1)$. Let $X_S$ denote the number of such sets of size at most $s_0 = n\e^2/(20c)$.
\beq{eq4}{
	\begin{aligned}
		\E\left(X_S\right) & \leq \sum_{s=4}^{s_0} \binom{n}{s} \Pr\left([s] \text{ contains $\frac{\e\D s}{2}$ edges}\right) \\
		& \leq \sum_{s=4}^{s_0} \left(\frac{ne}{s}\right)^s \E\left(\text{number of sets of $\frac{\e\D s}{2}$ or more edges in [s]}\right)\\
		& \leq \sum_{s=4}^{s_0} \left(\frac{ne}{s}\right)^s \binom{\binom{s}{2}}{\e\D s/2} \left(\frac{c}{n}\right)^{\e\D s/2} \\
		& \leq \sum_{s=4}^{s_0} \left(\frac{ne}{s}\right)^s \left(\frac{es^2c}{\e\D sn}\right)^{\e\D s/2} \\
		& \leq \sum_{s=4}^{s_0} \left(\frac{ne}{s}\right)^s \left(\frac{e\e}{20\D }\right)^{\e\D s/2}. \\
	\end{aligned}
} 
Let $u_s = \left(\frac{ne}{s}\right)^s \left(\frac{e\e}{20\D }\right)^{\e\D s/2}$. If $s\leq \log^2 n$, then $u_s \leq n^{-1/2}$ implying $\sum\limits_{s=4}^{\log^2 n} u_s \leq \frac{\log ^2n}{n^{1/2}} =o(1)$. If $s\geq \log^2 n$, then $u_s \leq (e/20)^{\log^2 n}$ so $\sum\limits_{s=\log^2 n}^{s_0} u_s \leq n (e/20)^{\log^2 n} = o(1)$. Thus, $e(S)\leq \frac{\e\D}{2} |S|$ for all $S$ w.h.p.

\begin{figure}[htbp]
	\centering
	
	\resizebox{0.75\textwidth}{!}
	{\begin{tabular}{cc}
		\begin{tikzpicture}[framed, scale=1]
			
			%First row G
			\draw (-3,0) ellipse (1.5 and 1);
			\node at (-3,-1.5) {$S$};
			
			\node[fill=black, circle, inner sep=1.5pt] (u1) at (-3.75,0) {};
			\node at (-4,-0.25) {$u$};
			\node[fill=black, circle, inner sep=1.5pt] (w1) at (-2.25,0) {};
			\node at (-2,-0.25) {$w$};
			\node[fill=black, circle, inner sep=1.5pt] (v1) at (-3,1.5) {};
			\node at (-3,1.8) {$v$};
			
			\draw (u1) -- (v1);
			\draw (w1) -- (v1);
			
			%First row G_2
			\draw (2,0) ellipse (1.5 and 1);
			\node at (2,-1.5) {$S$};
			
			\node[fill=black, circle, inner sep=1.5pt] (u2) at (1.25,0) {};
			\node at (1,-0.25) {$u$};
			\node[fill=black, circle, inner sep=1.5pt] (w2) at (2.75,0) {};
			\node at (3,-0.25) {$w$};
			\node[fill=black, circle, inner sep=1.5pt] (v2) at (2,1.5) {};
			\node at (2,1.8) {$v$};
			
			\draw (u2) -- (v2);
			\draw (w2) -- (v2);
			\draw[thick, red] (u2) -- (w2);
			
			%First row arrow
			\draw[->, thick, double, gray] (-1,0) -- (0,0);
			\node at (-0.5, 1.8) {Type 1};
		\end{tikzpicture} \ & \ 
		
		\begin{tikzpicture}[framed, scale=1]
				
			%Second row G
			\draw (1,0) ellipse (1.5 and 1);
			\node at (1,-1.5) {$S$};
			
			\node[fill=black, circle, inner sep=1.5pt] (u3) at (0.25,-0.3) {};
			\node at (0,-0.55) {$u$};
			\node[fill=black, circle, inner sep=1.5pt] (w3) at (1.75,-0.3) {};
			\node at (2,-0.55) {$w$};
			\node[fill=black, circle, inner sep=1.5pt] (v3) at (1,0.5) {};
			\node at (1,0.75) {$v$};
			
			\draw (u3) -- (v3);
			\draw (w3) -- (v3);
			
			%Second row G_2
			\draw (6,0) ellipse (1.5 and 1);
			\node at (6,-1.5) {$S$};
			
			\node[fill=black, circle, inner sep=1.5pt] (u4) at (5.25,-0.3) {};
			\node at (5,-0.55) {$u$};
			\node[fill=black, circle, inner sep=1.5pt] (w4) at (6.75,-0.3) {};
			\node at (7,-0.55) {$w$};
			\node[fill=black, circle, inner sep=1.5pt] (v4) at (6,0.5) {};
			\node at (6,0.75) {$v$};
			
			\draw (u4) -- (v4);
			\draw (w4) -- (v4);
			\draw[thick, red] (u4) -- (w4);
			
			%Second row arrow
			\draw[->, thick, double, gray] (3,0) -- (4,0);
			\node at (3.5, 1.8) {Type 2};
			
		\end{tikzpicture}
	\end{tabular}}
	
	\caption{The two types of edges introduced in a set $S$ when squaring graph $G$}
	\label{fig3}
	
\end{figure}
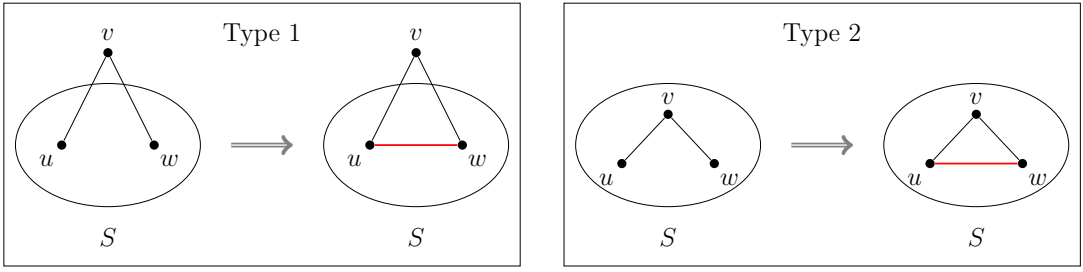

We now focus on the `new' edges in $G_2$, which were not present in $G_{n,p}$. For $u,w \in S$, a new edge $uw$ appears in $S$ when squaring $G_{n,p}$ only due to the existence of a common neighbor $v$. From containment of such a neighbor in $S$, the new edges in $G_2$ can be of two types seen in Figure \ref{fig3}.  \vspace{-15pt}

\begin{enumerate}[\qquad Type 1.]
	\setlength{\itemsep}{0pt}
	\item $v\notin S$, bounded by Lemma \ref{lem4}
	\item $v\in S$, bounded by Lemma \ref{lem5}
\end{enumerate} \vspace{-15pt}

If $S\cap W_\e=\emptyset$ then a vertex outside $S$ has at most $\e\D$ neighbors in $S$.  For a fixed set $S$ let $a_{k,S}$ denote the number of (vertex, set) pairs $(v,T)$ where $v\notin S$ and $T=N(v)\cap S$ with $|T|=k$.

\begin{lemma}\label{lem4}
	The following holds w.h.p. Let $A_1(S)=\sum_{k\leq \e\D}a_{k,S}k^2$ bound the number of $G_2$-edges of type 1 in $S$. Then $A_1(S)\leq 5\e\D |S|$ for all $S$ with $|S|\leq n/(10ck_0)$.
\end{lemma}

We now deal with the number of edges of type 2, i.e., $uw$ coming from paths $uvw$ of length two where $\set{u,v,w}\subseteq S$. Denote the number of such paths by $A_2(S)$. 

\begin{lemma}\label{lem5}
	W.h.p., $A_2(S)\leq (2c+1/2)\e\D|S|$ for all $S$ with $|S|\leq n/(10ck_0)$.
\end{lemma}
We provide proofs of the above two lemmas in Section \ref{sec3}. 

It follows from lemmas \ref{lem4} and \ref{lem5} and the upper bound on $e(S)$ obtained using \ref{eq4}, we have w.h.p.,
\beq{eq2}{
e(S) + A_1(S)+A_2(S)\leq (6+2c)\e\D|S|.
}
From \ref{eq1} and \ref{eq2}, we have established that any $S\seq [n]\sm W_\e$ contains at most $(6+2c)\e\D|S|$ edges w.h.p. 

\begin{corollary}\label{cor3}
	The vertices of $[n]\setminus W_\e$ can be list-colored with lists of at most $(6+2c)\e\D+1$ colors. 
\end{corollary}
\begin{proof} 
	We derive the required result by proving a more general statement with the principle of strong mathematical induction: for any vertex set $V$, if every subset $S\seq V$ contains at most $r|S|$ edges then $V$ is $(r+1)$-list-colorable. Clearly if $|V|\leq r+1$ then the result holds as every vertex has at most $r$ neighbors. Assume the result holds for all vertex sets of size $1,2,..,t$ for some $t\in \N$. For a set $V$ of size $|V|=t+1$, we have at most $r(t+1)$ edges inside. Then, there exists a vertex $v$ of degree $\leq r$. The induced subgraph on $V\sm\{v\}$ has vertex set of size $t$ with the required property on edges, and hence is $(r+1)$-list-colorable by the induction hypothesis. We assign colors to all the vertices in $V\sm \{v\}$ first. Now since $v$ has degree $\leq r$, there is at least one color available for $v$ from its list of size $r+1$, which is not used by any neighbors in $V\sm\{v\}$.
\end{proof}

\section{List coloring of \texorpdfstring{$G_2$}{} and proofs of lemmas} 
\label{sec3}

Given the above results on the number of neighbors, we can list-color $G_2$ as follows:
\begin{enumerate}[(1)]
	
	\item We list-color $V_{\e}$ with lists of size $q=\D(1+3\th^{1/3})$ colors. We do this greedily by arbitrarily ordering the vertices in $V_\e$ and coloring a vertex with the lowest index color available. Corollary \ref{cor1} implies that any vertex $v$ has at most $\D\brac{1+2\th^{1/3}}$ $G_2$-neighbors in $V_\e$ and so there will be an unused color.
	
	\item We list-color $W_{\e}\setminus V_\e$ with lists of size $q=\D(1+3\th^{1/3})$ greedily, i.e., we arbitrarily order the vertices in $W_{\e}\setminus V_\e$ and color a vertex with the lowest index color available. Corollary \ref{cor2} implies that any vertex $v\notin V_\e$ has at most $\D\brac{1+2\th^{1/3}}$ $G_2$-neighbors in $W_\e$ and so there will be an unused color. 
	
	\item We then list-color $[n]\setminus W_\e$ with at most $\D\brac{1+2\th^{1/3}+(6+2c)\th^{1/2}}+1$ colors greedily. This follows similarly from Corollaries \ref{cor2} and \ref{cor3}. Since $\th = o(1)$, the size of lists is bounded above by $q = \D(1+3\th^{1/3})$.
	
\end{enumerate}

\subsection{Proof of Lemma \ref{lem1}}

\begin{proof}
	To bound the probability of having $dist(v,w)<10$ for two vertices, we can count all paths of length $k$ between them in $\binom{n}{k} k!$ ways, for $k=1,2...,9$. We also need $v,w\in V_{2/3}$, so we can multiply by square of the probability of a vertex having a high degree. Let $\ell_0=2\D/3-10$. We have
	\begin{align*}
		\Pr(\exists v,w\in V_{2/3}:dist(v,w)<10)&\leq \sum_{k=1}^9\binom{n}{k}k!p^{k-1}\brac{\sum_{\ell=\ell_0}^{n-1}\binom{n}{\ell}p^\ell(1-p)^{n-10-\ell}}^2\\
		&\leq \sum_{k=1}^9nc^{k-1}n^{-4/3+o(1)}=o(1).
	\end{align*}
\end{proof}

\subsection{Proof of Lemma \ref{lem2}}

\begin{proof} There are $\binom{n}{s}$ choices for $S$, where $|S|=s\in[m,3m]$. There are at most $s^{s-2}$ choices for a spanning tree of $S$. We can choose the vertices of large degree in $\binom{s}{m}$ ways. Let $D$ be the sum of degrees of the chosen $m$ vertices. The probability that these vertices have large degrees can be bounded by a product, so
	\begin{align*}
		\Pr(\exists S)&\leq \sum_{m=2}^{2/\e}\sum_{s=m}^{3m}\binom{n}{s}s^{s-2}p^{s-1}\binom{s}{m}\sum_{D\geq (1+\th^{1/3})\D}\sum_{\ell_1+\cdots+\cdots\ell_m=D}\prod_{i=1}^m\brac{\sum_{k=\ell_i}^{n-s}\binom{n-s}{k}p^k\brac{1-p}^{n-s-k}}\\ 
		&\leq \sum_{m=2}^{2/\e}\sum_{s=m}^{3m}\binom{n}{s}s^{s-2}p^{s-1}2^s\sum_{D\geq (1+\th^{1/3})\D}\sum_{\ell_1+\cdots+\cdots\ell_m=D}\prod_{i=1}^m n^{-\ell_i/\D+O(\th)}\\
		&\leq \frac{2n}{c}\sum_{m=2}^{2/\e}\bfrac{2ec}{3m}^{3m}\sum_{D\geq (1+\th^{1/3})\D}\binom{D-1}{m-1}n^{-D+O(\th m)}\\
		&\leq \frac{2n}{c}\sum_{m=2}^{2/\e}\bfrac{2ec}{3m}^{3m}\bfrac{(1+\th^{1/2})\D}{m}^mn^{-(1+\th^{1/3}-O(\th^{1/2}))}\\
		&\leq n^{1+o(\th^{1/2})-(1+\th^{1/3}-O(\th^{1/2}))}=o(1).
	\end{align*}
\end{proof}

\subsection{Proof of Lemma \ref{lem3}}
\begin{proof}
	Let $d(i)$ denote the degree of vertex $i$ in $G_{n,p}$. The expected number of edges in $G_2$ is 
	\[
	\E\brac{\sum_{i=1}^n\frac{d(i)(d(i)+1)}{2}}=\frac{n}2\sum_{j=1}^{n-1}j(j+1)\binom{n-1}{j}p^j(1-p)^{n-1-j}=\frac{c^2(n-1)(n-2)+cn^2}{2n}.
	\]
	To show concentration around the mean, we use the following theorem from Warnke \cite{War}:
	\begin{theorem}\label{warnke}
		Let $X=(X_1,X_2,\ldots,X_N)$ be a family of independent random variables with $X_k$ taking values in a set $\Lambda_k$. Let $\Omega=\prod_{k\in[N]}\Lambda_k$ and suppose that $\Gamma\subseteq \Omega$ and $f:\Omega\to{\bf R}$ are given. Suppose also that whenever $\bx,\bx'\in \Omega$ differ only in the $k$-th coordinate 
		\[
		|f(\bx)-f(\bx')|\leq \begin{cases}c_k&if\ \bx\in\Gamma.\\d_k&otherwise.\end{cases}
		\]
		If $W=f(X)$, then for all reals $\gamma_k>0$,
		\[
		\Pr(W\geq \E(W)+t)\leq \exp\set{-\frac{t^2}{2\sum_{k\in[N]}(c_k+\gamma_k(d_k-c_k))^2}}+\Pr(X\notin \Gamma)\sum_{k\in [N]}\gamma_k^{-1}.
		\]
	\end{theorem}
	
	We use Theorem \ref{warnke} with $N=n$, defining $X_i=\set{j<i:\set{j,i}\text{ is an edge of }G_{n,p}}$ for $i=1,2,\ldots,n$, letting $f(X) = W =\sum_{i=1}^n\frac{d(i)(d(i)+1)}{2}$ be the total number of edges in $G_{n,p}^2$, and $\G=\set{\D(G_{n,p})\leq \log n}$. The condition of $\bx, \bx'$ differing only in the $k$-th coordinate can be seen as missing an edge. In this case, we can assign $c_k=\log^2n$, $d_k=n^2$ and $\Pr(X\notin\G)\leq (\log n)^{-\tfrac12\log n}$. Then we can take $\g_k=n^{-4}$ for $k\in [n]$ and $t=n^{2/3}$ to complete the proof of Lemma \ref{lem3}.
\end{proof}

\subsection{Proof of Lemma \ref{lem4}}
\begin{proof}
	Let $|S|\leq n/(10ck_0)$. We will prove that w.h.p.:
	
	\begin{enumerate}[(a)]
		\item For $k_0<k_1<k_2\leq \e\D$, for all sets $S\subseteq [n]\setminus V_\e$, 
		\[
		\sum_{k=k_1}^{k_2}a_{k,S}\leq \frac{(1+\e)|S|}{k_1}.
		\]
		\item $a_{k,S}\leq (10c)^{k_0}|S|$ for $k\leq k_0$.
	\end{enumerate}
	
	We have, where  $M_{u,k_1,k_2}=\set{(x_2,\ldots,x_{\e\D}):\sum_{k=k_1}^{k_2}x_k=u}$.
	\begin{align*}
		&\Pr\brac{\exists S,|S|=s\leq n/(10ck_0):\sum_{k=k_1}^{k_2}a_{k,S}\geq t}\\
		&\leq\sum_{s=k_1^{1/2}}^{n/(10ck_0)}\binom{n}{s}\sum_{u\geq t}\sum_{\bx\in M_{u,k_1,k_2}}   \binom{n}{x_{k_1},\ldots,x_{k_2},n-u}\prod_{k=k_1}^{k_2}\brac{\binom{s}{k}\bfrac{c}{n}^{k}}^{x_k}\\
		&\leq \sum_{s=k_1^{1/2}}^{n/(10ck_0)}\bfrac{ne}{s}^s\sum_{u\geq t}\sum_{\bx\in M_{u,k_1,k_2}} \binom{n}{x_{k_1},\ldots,x_{k_2},n-u} \prod_{k=k_1}^{k_2}\bfrac{sec}{k_1n}^{k_1x_k}\\
		&=\sum_{s=k_1^{1/2}}^{n/(10ck_0)}\bfrac{ne}{s}^s\sum_{u\geq t}\bfrac{sec}{k_1n}^{k_1u}\sum_{\bx\in M_{u,k_1,k_2}}\binom{n}{x_{k_1},\ldots,x_{k_2},n-u}\\
		&\leq \sum_{s=k_1^{1/2}}^{n/(10ck_0)}\bfrac{ne}{s}^s\sum_{u\geq t}\bfrac{sec}{k_1n}^{k_1u}\binom{n}{u}.
	\end{align*}
	
	Putting $t=(1+\e)s/k_1$, we have, for large $k_1$ i.e. for $k_1>2/\e$,
	\begin{align*}
		&\Pr\brac{\exists S,|S|=s\leq n/(10ck_0):\sum_{k=k_1}^{k_2}a_{k,S}\geq \frac{(1+\e)s}{k_1}}\\
		&\leq\sum_{s=k_1^{1/2}}^{n/(10ck_0)}\bfrac{ne}{s}^s\sum_{u\geq t}\bfrac{sec}{k_1n}^{k_1u}\binom{n}{u}\\
		&\leq 2\sum_{s=k_1^{1/2}}^{n/(10ck_0)}\bfrac{ne}{s}^s\bfrac{sec}{k_1n}^{(1+\e)s}\binom{n}{(1+\e)s/k_1}\\
		&\leq 2\sum_{s=k_1^{1/2}}^{n/(10ck_0)}\bfrac{ne}{s}^s\bfrac{sec}{k_1n}^{(1+\e)s}\bfrac{nek_1}{(1+\e)s}^{(1+\e)s/k_1}\\
		&=2\sum_{s=k_1^{1/2}}^{n/(10ck_0)}\brac{\bfrac{s}{n}^{\e-(1+\e)/k_1}\frac{e^{2+\e+(1+\e)/k_1}c^{1+\e}}{k_1^{(1+\e)(k_1-1)/k_1}}}^s
		=o(n^{-2}),
	\end{align*}
verifying (a).	

	We also have
	\begin{align}
		\Pr(\exists S,|S|=s\leq n/(10ck_0):a_{k,S}\geq t)&\leq \binom{n}{s}\binom{n}{t}\brac{\binom{s}{k}\bfrac{c}{n}^{k}}^{t}\nn\\
		&\leq \bfrac{ne}{s}^s\bfrac{ne}{t}^t\cdot \brac{\frac{se}{k}\cdot\frac{c}{n}}^{kt}.\label{eq0}
	\end{align}
	
	We put $t=(10c)^{k_0}s$ for $2\leq k<k_0$  . Then we have, where $L=(10c)^{k_0}$,
	
	\begin{align}
		\Pr(\exists S,|S|=s\leq n/(10ck_0):2\leq k\leq k_0,a_{k,S}\geq  Ls)&\leq \bfrac{ne}{s}^s\cdot\bfrac{ne}{Ls}^{Ls}\cdot \bfrac{sec}{kn}^{Lks}\nn\\
		&=\brac{\bfrac{s}{n}^{L(k-1)-1}\cdot e\cdot\bfrac{e}{L}^L\cdot\bfrac{ec}{k}^{Lk}}^s\nn\\
		&\leq \brac{\bfrac{s}{n}^{L(k-1)-1}\cdot e^{cL+2}}^s.\label{b1}
	\end{align}
	
	If $k\geq 3$ then the bracketed term $\s_s$ in \eqref{b1} is at most $\bfrac{se^c}{n}^L=o(1)$ and so $\sum_{s\geq 1}\s_s^s=o(1)$. If $k=2$ we write $\s_s=\bfrac{se^c}{n}^{L-1}\cdot e^{c+2}=o(1)$, as well. This verifies (b).
	
	So, by dividing $[1,\e\D]$ into intervals of size $\e\D/2^i$ for $i\geq1$, we get that for all $|S|\leq n/(10ck_0)$,	
	\begin{align}
		A_1(S)
		& \leq \sum_{k=2}^{k_0}(10c)^{k_0}k^2s+\sum_{i\geq 1}\frac{(1+\e)s}{\e\D/2^i}\cdot\frac{(\e\D)^2}{2^{2i-2}}\nn\\
		& \leq (2/\e)^{3}(10c)^{2/\e}s+4(1+\e)\e\D s \\ 
		& \leq 5\e\D s.\label{eq2a}
	\end{align}
\end{proof} 

\subsection{Proof of Lemma \ref{lem5}}
\begin{proof}
	The number of such paths is equal to $\sum_{v\in S}\frac{d_S(v)(d_S(v)+1)}{2}$ where $d_S(v)$ is the degree of $v$ in $S$. 
	\beq{eq3}{
		A_2(S)=\sum_{v\in S,d(v)\leq\e\D}\frac{d_S(v)(d_S(v)+1)}{2}\leq \e\D\sum_{v\in S}\frac{d_S(v)+1}2=\e\D(e(S)+|S|/2),
	}
	where $e(S)$ is the number of $G_{n,p}$ edges entirely contained in $S$. Now
	\mults{
		\Pr(\exists S,s=|S|\leq n/(10ck_0):e(S)\geq 2cs)\leq\sum_{s=2}^{n/(10ck_0)} \binom{n}{s}\binom{\binom{s}{2}}{2cs}\bfrac{c}{n}^{2s}\leq \sum_{s=2}^{n/(10ck_0)}\brac{\frac{ne}{s}\cdot \bfrac{se}{4n}^{2}}^s\\
		=\sum_{s=2}^{n/(10ck_0)}\brac{\frac{s}{n}\cdot\frac{e^3}{16}}^s=o(n^{-1}).
	}
	So,
	\[
	A_2(S)\leq \e\D(2c+1/2) |S|\text{ for all }|S|\leq n/(10ck_0),\  w.h.p.
	\]
\end{proof}

\section{Conclusions}\label{4}
While we have shown that $\chi_\ell(G_2)\sim \D$ w.h.p., it is possible that $\chi(G_2)= \D+1$ w.h.p. This would be quite pleasing, but we are not confident enough to make this a conjecture. It is of course interesting to further consider $\chi(G_2)$ or $\chi_\ell(G_2)$ when $np\to\infty$. Note that when $np\gg n^{1/2}$, the diameter of $G_{n,p}$ is equal to 2 w.h.p., in which case $G_2=K_n$. One can also consider higher powers of $G_{n,p}$ as was done in \cite{AF} and \cite{KLMP}. Such considerations are more technically challenging, especially with list coloring.


\begin{thebibliography}{99}
\bibitem{AN} D. Achlioptas and A. Naor, The two possible values of the chromatic number of a random graph, {\em Annals of Mathematics} 162 (2005) 1335-1351.
\bibitem{AKS} N. Alon, M. krivelevich and B. Sudakov, List coloring of random and pseudo-random graphs, {\em Combinatorica} 19 (1999) 453-472.
\bibitem{CV}  A. Coja--Oghlan and D. Vilenchik, Chasing the $k$-colorability threshold, Proceedings of the 54th IEEE Annual Symposium on Foundations of Computer Science (2013).
\bibitem{AF} A.M. Frieze and G. Atkinson, On the $b$-independence number of sparse random graphs, {\em Combinatorics, Probability and Computing} 13 (2003), 295-310.
\bibitem{FK} A.M. Frieze and M. Karo\'nski, Introduction to Random Graphs, Cambridge University Press (2015).
\bibitem{KLMP} K. Garapaty, D. Lokshtanov, H. Maji and A. Pothen, The Chromatic Number of Squares Of Random Graphs, {\em Journal of Combinatorics} 14 (2023) 507-537.
\bibitem{L} T. {\L}uczak, The chromatic number of random graphs,
{\em Combinatorica} 11 (1991) 45-54.
\bibitem{War} L. Warnke, On the Method of Typical Bounded Differences, {\em Combinatorics, Probability and Computing} 25 (2016) 269–299.
\end{thebibliography}
\end{document}